\documentclass[10pt]{article}

\usepackage{amsthm,amsmath,amssymb}

\usepackage{graphicx,url}

\usepackage[colorlinks=true,citecolor=blue,linkcolor=blue,urlcolor=blue]{hyperref}

\newcommand{\arxiv}[1]{\href{http://arxiv.org/abs/#1}{\texttt{arXiv:#1}}}

\theoremstyle{plain}
\newtheorem{theorem}{Theorem}[section]
\newtheorem{lemma}[theorem]{Lemma}
\newtheorem{corollary}[theorem]{Corollary}
\newtheorem{proposition}[theorem]{Proposition}

\theoremstyle{definition}
\newtheorem{definition}[theorem]{Definition}
\newtheorem{example}[theorem]{Example}

\theoremstyle{remark}

\title{\bf The Redei-Berge function in noncommuting variables }

\author{Stefan Mitrovi\'c\\
\small Faculty of Mathematics\\[-0.8ex]
\small University of Belgrade\\[-0.8ex]
\small Serbia\\
\small\tt stefan.mitrovic@matf.bg.ac.rs\\
}

\begin{document}

\maketitle

\begin{abstract}
Recently, Stanley and Grinberg introduced a symmetric function associated to digraphs, called the Redei-Berge symmetric function. This function, however, does not satisfy the deletion-contraction property, which is a very powerful tool for proving various identities using induction. In this paper, we introduce an analogue of this function in noncommuting variables which does have such property. Furthermore, it specializes to the ordinary Redei-Berge function when the variables are allowed to commute. This modification allows us to further generalize properties that are already proved for the original function and to deduce many new ones.

\bigskip\noindent \textbf{Keywords}: digraph, Redei-Berge symmetric function, deletion-contraction, noncommutative symmetric function

\small \textbf{MSC2020}: 05C20, 05E05
\end{abstract}

\section{Introduction}

In 2022, Stanley and Grinberg defined in \cite{S} a symmetric function associated to digraphs and named it the Redei-Berge function, in honor of two mathematicians, whose results about the number of Hamiltonian paths in a digraph they managed to deduce in a new way - using the theory of symmetric functions. Berge's theorem says that if $X$ is a simple digraph and $\overline{X}$ is its complement, then the number of Hamiltonian paths of $\overline{X}$ is congruent to the same number for $X$ modulo $2$ \cite{B}. Redei's theorem says that if $X$ is a tournament, then the number of Hamiltonian paths of $X$ is odd \cite{LR}.

The first version of the Redei-Berge function appeared in 1996, in Chow's paper \cite{C}, and later in Wiseman's paper \cite{W}. In \cite{GS}, it is shown that this function is the image of the isomorphism class of digraph under a certain canonical Hopf algebra morphism. This reconceptualization is used to prove various properties of the Redei-Berge function, including the deletion-contraction property for its principal evaluation - the Redei-Berge polynomial. Unfortunately, unlike its principal evaluation, the Redei-Berge function does not have the deletion-contraction property. However, the authors in \cite{MS} gave some decomposition techniques that could serve as a replacement of this property.

In this paper, inspired by \cite{SG}, we define an analogue of this symmetric function in noncommuting variables. The reason for not letting the variables commute is the same as in \cite{SG}, where the authors managed to deduce many results about the original chromatic function by observing its analogue in noncommuting variables. Namely, although this modification may seem like an additional complication, it allows us to keep track of each vertex, which makes it possible for us to obtain the deletion-contraction property. This is a powerful tool for proving various identities and for deriving generalizations of results about the original Redei-Berge function from \cite{S}, \cite{GS} and \cite{MS} using induction.

\section{Preliminaries}

The algebra of quasisymmetric functions consists of formal power series of bounded degree in variables $x_1, x_2, \ldots$ with coefficients in $\mathbb{Q}$ such that the coefficient of the monomial $x_1^{a_1}x_2^{a_2}\cdots x_k^{a_k}$ is equal to the coefficient of the monomial $x_{i_1}^{a_1}x_{i_2}^{a_2}\cdots x_{i_k}^{a_k}$ whenever $i_1<i_2<\cdots<i_k$. Fundamental quasisymmetric functions given by \begin{equation}\label{fundamental}
F_I=\sum_{\substack{1\leq i_1\leq i_2\leq\cdots\leq i_n\\
                  i_j<i_{j+1} \ \mathrm{for \ each} \ j\in I}}x_{i_1}x_{i_2}\cdots x_{i_n}, \quad I\subseteq[n-1]
\end{equation} form a basis of this algebra. For basics of quasisymmetric functions, see \cite{EC}.

The algebra of symmetric functions consists of quasisymmetric functions that are invariant under the action of permutations on the set of variables. As a vector space, it has many natural bases, see \cite{EC}. One such basis of this space consists of monomial symmetric functions. For partition $\lambda=(\lambda_1, \lambda_2, \ldots, \lambda_n)$ of some integer, monomial symmetric function $m_\lambda$ is defined as
\[m_{\lambda}=\sum x_{i_1}^{\alpha_1}x_{i_2}^{\alpha_2}\cdots x_{i_n}^{\alpha_n},\]
where the sum runs over all distinct permutations $\alpha=(\alpha_1, \alpha_2, \ldots, \alpha_n)$ of $\lambda$ and over all sequences $i_1<i_2<\cdots<i_n$ of positive integers.

We will be particularly interested in the power sum basis. The $i$th power sum symmetric function is defined as
\[p_0=1 \hspace{5mm }\textrm{ and } \hspace{5mm}p_i=\sum_{j=1}^{\infty}x_j^i\]
for $i\geq 1$. For partition $\lambda=(\lambda_1, \lambda_2, \ldots, \lambda_k)$, we define 

\[p_{\lambda}=p_{\lambda_1}p_{\lambda_2}\cdots p_{\lambda_k}.\]

Finally, the third basis of our interest consists of elementary functions. The $i$th elementary symmetric function is given by 
\[e_0=1\hspace{5mm} \textrm{ and } \hspace{5mm}e_i=\sum_{j_1<j_2<\cdots<j_i}x_{j_1}x_{j_2}\cdots x_{j_i} \] for $i\geq 1$.
For partition $\lambda=(\lambda_1, \lambda_2, \ldots, \lambda_k)$, we define
\[e_{\lambda}=e_{\lambda_1}e_{\lambda_2}\cdots e_{\lambda_k}.\] 

A digraph $X$ is a pair $X=(V,E)$, where $V$ is a finite set and $E$ is a collection
$E\subseteq V\times V$. Elements $u\in V$ are vertices and
elements $(u,v)\in E$ are directed edges of the digraph $X$. If $V$ has $n$ elements, a $V$-listing is a list of all vertices with no repetitions, i.e. a bijective map $\sigma:[n]\rightarrow V$. We write $\Sigma_V$ for the set of all $V$-listings. For a $V$-listing $\sigma=(\sigma_1,\ldots,\sigma_n)\in\Sigma_V$, define the $X$-descent set as

\[X\mathrm{Des}(\sigma)=\{1\leq i\leq n-1 | (\sigma_i,\sigma_{i+1})\in E\}.\] 

Grinberg and Stanley associated to a digraph $X$ a generating function for $X$-descent sets and named it the Redei-Berge symmetric function, see \cite{S}.

\begin{equation}\label{descents}
U_X=\sum_{\sigma\in\Sigma_V}F_{X\mathrm{Des}(\sigma)}.
\end{equation}

If $X=(V, E)$ is a digraph, its complementary digraph is the digraph $\overline{X}=(V, (V\times V)\setminus E)$ and its opposite digraph is $X^{op}=(V, E')$, where $E'=\{(v, u) \ \mid \ (u, v)\in E\}$.

\begin{definition}
    Let $X=(V, E)$ be a digraph and let $\mathbb{S}_V$ be the group of permutations of $V$. Then, we define
    \[\mathbb{S}_V(X)=\{\sigma\in \mathbb{S}_V\mid \textrm{ each non-trivial cycle of }\sigma \textrm{ is a cycle of } X\},\]
    \[\mathbb{S}_V(X, \overline{X})=\{\sigma\in \mathbb{S}_V\mid \textrm{ each cycle of }\sigma \textrm{ is a cycle of } X \textrm{ or a cycle of }\overline{X}\}.\]
\end{definition}

 For a permutation $\sigma$, let $\mathrm{type}(\sigma)$ denote the partition whose entries are the lengths of the cycles of $\sigma$. 

\begin{theorem}\cite{S}\label{pbaza} If $X=(V, E)$ is a digraph, then \[U_X=\sum_{\sigma\in\mathbb{S}_V(X, \overline{X})}(-1)^{\varphi(\sigma)}p_{\mathrm{type}(\sigma)},\]
with $\varphi(\sigma)=\sum_{\gamma}(\ell (\gamma)-1)$, where the summation runs over all cycles $\gamma$ of $\sigma$ that are cycles in $X$ and $\ell(\gamma)$ denotes the length of the cycle $\gamma$. Consequently, $U_X$ is a symmetric function.
    \end{theorem} 

We say that a loopless digraph $X=(V, E)$ is a tournament if for every two distinct vertices $u, v\in V$ exactly one of $(u, v)$ and $(v, u)$ is an edge of $X$. Previous theorem is used in \cite{S} to prove the following corollary.

\begin{corollary}
    \cite{S} \label{turnir} If $X=(V, E)$ is a tournament, then \[U_X=\sum_{\substack{\sigma\in \mathbb{S}_V(X)\\ \textrm{all cycles of }\sigma \textrm{ have odd length}}}2^{\psi(\sigma)}p_{\mathrm{type}(\sigma)},\]
    where $\psi(\sigma)$ denotes the number of nontrivial cycles of $\sigma$.
\end{corollary}

The authors in \cite{GS} gave another interpretation of this function, which will be more convenient for defining our generalization.

\begin{definition} \label{prijateljski}
Let $X=(V,E)$ be a digraph. For a coloring of vertices with positive integers $f:V\rightarrow\mathbb{P}$, a $V$-listing $\sigma=(\sigma_1,\ldots,\sigma_n)\in\Sigma_V$ is called $(f,X)$-{\it friendly} if \[f(\sigma_1)\leq f(\sigma_2)\leq\cdots\leq f(\sigma_n) \ \text{and}\]  \[f(\sigma_j)<f(\sigma_{j+1}) \ \mathrm{for \ each} \ j\in[n-1] \ \mathrm{satisfying} \ (\sigma_j,\sigma_{j+1})\in E.\]
\end{definition}

Denote by $\Sigma_V(f,X)$ the set of all $(f,X)$-friendly $V$-listings and by $\delta_f:\Sigma_V\rightarrow\{0,1\}$ its indicator function. For a coloring $f:V\rightarrow\mathbb{P}$ we write $\mathbf{x}_f=\prod_{v\in V}x_{f(v)}$.

\begin{theorem}\cite{GS}\label{nature}
The Redei-Berge symmetric function $U_X$ of a digraph $X=(V, E)$ satisfies
\[U_X=\sum_{f:V\rightarrow\mathbb{P}}\sum_{\sigma\in\Sigma_V}\delta_f(\sigma)\mathbf{x}_f.\]
\end{theorem}
For digraphs $X=(V,E)$ and $Y=(V',E')$ we define the product $X\cdot Y$ as the digraph on the disjoint union $V\sqcup V'$ with the set of directed edges $E\cup E'\cup\{(u,v)\ |\ u\in V, v\in V'\}.$ The following properties will be generalized in our paper.

\begin{theorem}\label{opozit}
    \cite{GS} For any digraph $X$, \[U_X=U_{X^{op}}.\] Consequently, if $X$ is a tournament, then \[U_X=U_{\overline{X}}.\]
\end{theorem}

\begin{theorem}
    \cite{GS}\label{proizvod}
    For any two digraphs $X$ and $Y$, \[U_{X\cdot Y}=U_X\cdot U_Y.\]
\end{theorem}

As we have noted in the introduction, the deletion-contraction property does not hold for the Redei-Berge function. Nevertheless, some decomposition techniques for $U_X$ that could replace deletion-contraction can be found in \cite{MS}.

\begin{theorem}
    \cite{MS}\label{razbijanje}  If $X=(V, E)$ is a digraph that is not a disjoint union of paths, then \begin{equation} \label{podskupovi}
          U_X=\sum_{S\subseteq E, S\neq\emptyset} (-1)^{|S|-1}U_{X\setminus S}.  \end{equation}
\end{theorem}

\begin{theorem}
    \cite{MS}\label{ciklus} If $e_1, e_2, \ldots, e_k$ is a list of edges that form a directed cycle in a digraph $X=(V, E)$, then \[U_X=\sum_{\substack{S\subseteq \{e_1, e_2, \ldots, e_k\} \\ S\neq\emptyset}}(-1)^{|S|-1}U_{X\setminus S}. \]
\end{theorem}

We now introduce the space of our particular interest - the vector space of the noncommutative symmetric functions. Noncommutative symmetric functions will be indexed by the elements of partition lattice. Let $\Pi_n$ denote the lattice of set partitions of $[n]=\{1, 2, \ldots, n\}$ ordered by refinement. For $\pi\in\Pi_n$, we write $\pi=B_1/B_2/\ldots/B_l$ if the $B_i$'s are the blocks of $\pi$. Let $\lambda(\pi)$ denote the integer partition of $n$ whose parts correspond to the block sizes of $\pi$. If $\lambda(\pi)=(1^{r_1}, 2^{r_2}, \ldots, n^{r_n})$, we will write $|\pi|$ for $r_1!r_2!\cdots r_n!$ and $\pi!$ for $1!^{r_1}2!^{r_2}\cdots n!^{r_n}$.

We define the noncommutative monomial symmetric function, $m_{\pi}$, by \[m_{\pi}=\sum_{i_1, i_2, \ldots, i_n}x_{i_1}x_{i_2}\cdots x_{i_n},\]
where the sum is over all sequences $i_1, i_2, \ldots, i_n$ of positive integers such that $i_j=i_k$ if and only if $j$ and $k$ are in the same block of $\pi$. It is easy to see that letting the variables commute transforms $m_\pi$ to $|\pi|m_{\lambda(\pi)}$. The noncommutative monomial symmetric functions are linearly independent over $\mathbb{C}$ and we call their span the algebra of noncommutative symmetric functions.

Another basis of this space consists of the noncommutative power sum symmetric functions given by \begin{equation}
    \label{pprekom}p_{\pi}=\sum_{\pi\leq\sigma}m_\sigma=\sum_{i_1, i_2, \ldots, i_n}x_{i_1}x_{i_2}\cdots x_{i_n},\end{equation}
where the second sum is over all positive integer sequences such that $i_j=i_k$ whenever $j$ and $k$ are in the same block of $\pi$. Clearly, if we let the variables commute, we transform $p_{\pi}$ into $p_{\lambda(\pi)}$.

Finally, the third basis we will be interested in contains the noncommutative elementary symmetric functions, defined by \[e_{\pi}=\sum_{i_{1}, i_{2}, \ldots, i_n}x_{i_1}x_{i_2}\cdots x_{i_n},\]
where the sum runs over all sequences $i_1, i_2, \ldots, i_n$ of positive integers such that $i_j\neq i_k$ if $j$ and $k$ are in the same block of $\pi$. Allowing the variables commute transforms $e_\pi$ into $\pi !e_{\lambda(\pi)}.$ In \cite{D}, it is shown that \begin{equation} \label{pprekoe}
    p_{\pi}=\frac{1}{\mu(\widehat{0}, \pi)}\sum_{\sigma\leq\pi}\mu(\sigma, \pi)e_{\sigma},
\end{equation}
where $\widehat{0}$ denotes the minimal element and $\mu$ denotes the Möbius function of $\Pi_n$.

 For basics of symmetric and quasisymmetric functions in noncommuting variables, see \cite{BZ}, \cite{SG}. Note that these functions are different from the noncommuting symmetric functions studied by Gelfand. 
 
 It is obvious that these functions are invariant under the usual action of the symmetric group, hence, these functions are symmetric in the usual sense. On the other hand, we will need to define another action of the symmetric group $\mathbb{S}_n$ on this space, which permutes the positions of the variables. For $\delta\in\mathbb{S}_n$, we define $\delta\circ(x_{i_1}x_{i_2}\cdots x_{i_n})=x_{i_{\delta^{-1}(1)}}x_{i_{\delta^{-1}(2)}}\cdots x_{i_{\delta^{-1}(n)}}$, which we extend linearly to the whole space of quasisymmetric functions in noncommuting variables. It is clear that $\delta\circ m_{\pi}=m_{\delta(\pi)}$ and that $\delta\circ p_{\pi}=p_{\delta(\pi)}$, where $\delta(\pi)$ denotes a partition obtained from $\pi$ by permuting the elements of the blocks of $\pi$ by $\delta$.

We define the induction operation on monomials, denoted by $\uparrow$, as \[(x_{i_1}x_{i_2}\ldots x_{i_{n-2}}x_{i_{n-1}})\uparrow=x_{i_1}x_{i_2}\ldots x_{i_{n-2}}x_{i_{n-1}}^2\] and extend it linearly to the space of noncommutative quasisymmetric functions.  We can easily see how this operation affects the monomial and the power sum symmetric functions. For $\pi\in\Pi_{n-1}$, let $\pi+(n)\in\Pi_n$ denote the partition $\pi$ with $n$ inserted into the block containing $n-1$. Then, $m_{\pi}\uparrow=m_{\pi+(n)}$ and $p_{\pi}\uparrow=p_{\pi+(n)}$.

\section{The Redei-Berge function in noncommuting variables}

Inspired by the work of Gebhard and Sagan \cite{SG}, in this section we define the central object of our paper - the noncommutative Redei-Berge symmetric function. As we have already mentioned, our definition will be motivated by Theorem \ref{nature}.

\begin{definition}
    For a digraph $X=(V, E)$ with vertices labeled $v_1, v_2, \ldots, v_n$ in fixed order, its Redei-Berge function in noncommuting variables, denoted as $W_X$, is \[W_X=\sum_{f:V\rightarrow\mathbb{P}}\sum_{\sigma\in\Sigma_V}\delta_f(\sigma)x_{f(v_1)}x_{f(v_2)}\cdots x_{f(v_n)}.\]
\end{definition}

From the definition, it is not quite clear that this expression gives us a symmetric function in noncommuting variables. However, it is obvious that $W_X$ is a quasisymmetric function in noncommuting variables. Furthermore, $W_X$ depends not only on $X$, but also on the labeling of its vertices. Clearly, if we let variables commute, we transform $W_X$ into $U_X$.

\begin{example}  Let $K_n=(V, V\times V)$ denote the complete digraph on $n$ vertices. For a function $f: V\rightarrow\mathbb{P}$, there is a friendly listing of $V$ if and only if $f(u)\neq f(v)$ for any two distinct vertices $u, v\in V$. Moreover, for such function, there is exactly one friendly listing. In other words, $W_{K_n}=e_{([n])}$, hence\[U_n=n!e_n.\] \end{example}

\begin{example}\label{diskretan}
    Let $D_n=(V, \emptyset)$ be the discrete digraph on $n$ vertices. For a function $f: V\rightarrow\mathbb{P}$ such that $f[V]$ contains $k$ values $c_1<c_2< \cdots< c_k$, there are $|f^{-1}[\{c_1\}]|!|f^{-1}[\{c_2\}]|!\cdots |f^{-1}[\{c_k\}]|!$ listings that are friendly with $f$. Therefore, \[W_{D_n}=\sum_{\pi\in\Pi_n} \pi! m_{\pi}.\] 
\end{example}

Certain properties of the Redei-Berge function in noncommuting variables can be directly deduced from its definition. The following proposition is a generalization of Theorem \ref{opozit}. We list the vertices of $X^{op}$ and $\overline{X}$ in the same order as in $X$.

\begin{theorem} For any labeled digraph $X=(V, E)$, \[W_X=W_{X^{op}}.\] If $X$ is a labeled tournament, then \[W_X=W_{\overline{X}}.\]
\end{theorem}

\begin{proof} If $\sigma=(\sigma_1, \sigma_2, \ldots, \sigma_n)$ is an $f$-friendly listing of $V$ in $X$, it induces a unique $f$-friendly listing of $V$ in $X^{op}$ in the following way. First, we need to list these vertices in non-decreasing order with regard to $f$. Furthermore, we list the vertices with the same $f$ value in $\sigma$-reversing order. This is because, if $f(\sigma_i)=f(\sigma_{i+1})$, then $(\sigma_i, \sigma_{i+1})$ is not an edge of $X$. Therefore, $(\sigma_{i+1}, \sigma_i)$ is not an edge of $X^{op}$, so the conditions from Definition \ref{prijateljski} are satisfied.
    
\end{proof}

If $X$ and $Y$ are digraphs with labeled vertices, then we label the vertices of $X\cdot Y$ by listing the vertices of $X$ first in the same order as in $X$ and then the vertices of $Y$ in the same order as in $Y$. The following is a generalization of Theorem \ref{proizvod}.

\begin{theorem} For any two labeled digraphs, \[W_{X\cdot Y}=W_X\cdot W_Y.\]    
\end{theorem}
\begin{proof}
    Let $X=(V, E)$ and $Y=(V', E')$ be labeled digraphs. If $f:V\rightarrow \mathbb{P}$ and $g:V'\rightarrow\mathbb{P}$ are two functions, then they induce a unique function, denoted as $fg:V\sqcup V'\rightarrow \mathbb{P}$ and vice-versa. Likewise, if $\sigma\in \Sigma_V(f,X)$ and $\sigma'\in \Sigma_{V'}(g,Y)$, then there is a unique listing in $\Sigma_{V\sqcup V'}(fg,X\cdot Y)$ that corresponds to the pair $(\sigma, \sigma')$. Namely, we need to list these vertices in non-decreasing order with regard to $fg$. Further, the vertices with the same $fg$ value, need to be arranged by listing the vertices from $Y$ first in $\sigma'$ order and then the vertices from $X$ in $\sigma$ order. This is because, if $v\in V$ precedes $v'\in V'$, then, since $(v, v')$ is an edge of $X\cdot Y$, we would need to have $fg(v)<fg(v')$. In the same manner, any listing in $\Sigma_{V\sqcup V'}(fg,X\cdot Y)$ uniquely splits to a pair of listings from $\Sigma_V(f,X)\times\Sigma_V(g, Y)$.
\end{proof}

In order to show that $W_X$ satisfies the deletion-contraction property, we need a distinguished edge. If the vertices of $X$ are labeled $v_1, v_2, \ldots, v_n$, we would like this edge to be exactly $(v_{n-1}, v_n)$. To obtain such a labeling, we will define an action of the symmetric group $\mathbb{S}_n$ on the set of vertices $v_1, v_2, \ldots, v_n$ by $\delta(v_i)=v_{\delta(i)}$ for $\delta\in\mathbb{S}_n$. Clearly, the following statement holds.

\begin{proposition} [Relabeling proposition] \label{relabel} For any digraph $X$, $\delta\circ W_X=W_{\delta(X)}$.
\end{proposition}

The deletion of an edge $e\in E$ from a digraph $X=(V,E)$ is the digraph $X\setminus e=(V,E\setminus\{e\})$. The contraction of $X$ by $e=(u,v)\in E$ is the digraph $X/e=(V', E')$, where $V'=V\setminus\{u,v\}\cup\{e\}$ and $E'$ contains all edges in $E$ with vertices different from $u,v\in V$ and additionally for $w\neq u,v$ we have
\begin{itemize}
\item $(w,e)\in E'$ if and only if $(w,u)\in E$ and
\item $(e,w)\in E'$ if and only if $(v,w)\in E$.
\end{itemize}
Note that the above operation of edge contraction on digraphs is different from the usual one since it takes into account the orientation of the contracting edge.

\begin{theorem} [Deletion-contraction] Let $X=(V, E)$ be any digraph with vertices labeled $v_1, v_2, \ldots, v_n$ and let $e=(v_{n-1}, v_n)$ be an edge in $X$. Then, \begin{equation}\label{delcon}
    W_X=W_{X\setminus e}-W_{X/e}\uparrow,\end{equation}
where the vertex obtained by contraction in $X/e$ is labeled $v_{n-1}$.
\end{theorem}

\begin{proof}
For any $f:V\rightarrow\mathbb{P}$, $\Sigma_V(f,X)\subseteq\Sigma_V(f,X\setminus e)$. On the other hand, the difference $\Sigma_V(f,X\setminus e)\setminus\Sigma_V(f,X)$ contains all $V$-listings $\sigma=(\sigma_1,\ldots,\sigma_n)$ with properties
\[f(\sigma_1)\leq\cdots\leq f(\sigma_n),\]
\[f(\sigma_i)<f(\sigma_{i+1}) \text{ for } (\sigma_i,\sigma_{i+1})\in E\setminus\{e\} \ \text{,}\]
\[(v_{n-1}, v_n)=(\sigma_j,\sigma_{j+1}) \ \text{for \ some} \ j=1,\ldots,n-1 \ \text{and} \ f(v_{n-1})=f(v_n).\] Such a listing determines the $V'$-listing \[\widehat{\sigma}=\left\{\begin{array}{ccc} (e,\sigma_3,\ldots,\sigma_n),& j=1,\\
(\sigma_1,\ldots,\sigma_{j-1},e,\sigma_{j+2},\ldots,\sigma_n),& 1<j<n-1,\\ (\sigma_1,\ldots,\sigma_{n-2},e),& j=n-1\end{array}\right.\] on the set of vertices $V'$ of the digraph $X/e$. Let $\widetilde{f}:V'\rightarrow\mathbb{P}$ be the coloring induced by $f$ with $\widetilde{f}(w)=\left\{\begin{array}{cc} f(w),& w\neq e,\\
f(v_{n-1})=f(v_n),& w=e\end{array}\right.$. Then $\widehat{\sigma}$ is a $(\widetilde{f}, X/e)$-friendly $V'$-listing. Any $(\widetilde{f},X/e)$-friendly $V'$-listing for some coloring $\widetilde{f}:V'\rightarrow\mathbb{P}$ is obtained uniquely in this way. Hence, the coefficients of $x_{f(v_1)}x_{f(v_2)}\cdots x_{f(v_n)}=x_{\widetilde{f}(v_1)}x_{\widetilde{f}(v_2)}\ldots x_{\widetilde{f}(v_{n-1})}\uparrow$ of the left and the right side of Equation (\ref{delcon}) are the same.
\end{proof}

According to Proposition \ref{relabel}, the deletion-contraction formula from the previous theorem can be applied to any digraph containing an edge that is not a loop. Inductively, we can remove all edges that are not loops. Since Example \ref{diskretan} tells us that discrete digraphs have symmetric Redei-Berge function and since loops do not affect this function, we get the following.

\begin{theorem}
    The Redei-Berge function in noncommuting variables of any digraph is symmetric.
\end{theorem}

The deletion-contraction property allows us to derive recurrence relations dealing with some types of familiar digraphs. Unfortunately, the Redei-Berge function does not behave especially nice over disjoint union of digraphs.

\begin{example}
    Let $P_n$ denote the path digraph with labeled vertices $v_1, v_2, \ldots, v_n$ and with edges $(v_i, v_{i+1})$ for $i\in [n-1]$. Applying deletion-contraction to $(v_{n-1}, v_n)$, we obtain that for $n\geq 2$ \[W_{P_{n}}=W_{P_{n-1}\cup D_1}-W_{P_{n-1}}\uparrow.\]

\end{example}

We are now able to prove the generalization of Theorem \ref{pbaza}. For a permutation $\sigma\in\mathbb{S}_n$, let $\mathrm{Type}(\sigma)$ be the partition of the set $\{1, 2, \ldots, n\}$ whose blocks correspond to the cycles of the unique cycle decomposition of $\sigma$. If $\{f_i\}$ is a basis of some vector space $V$ and $v\in V$, let $[f_i]v$ denote the coefficient of $f_i$ in the expansion of $v$ in this basis.

\begin{theorem} \label{razvoj}
    If $X=(V, E)$ is a digraph with labeled vertices $v_1, v_2, \ldots, v_n$, then \[W_X=\sum_{\sigma\in\mathbb{S}_V(X, \overline{X})}(-1)^{\varphi(\sigma)}p_{\mathrm{Type}(\sigma)},\]
with $\varphi(\sigma)=\sum_{\gamma}(\ell (\gamma)-1)$, where the summation runs over all cycles $\gamma$ of $\sigma$ that are cycles in $X$ and $\ell(\gamma)$ denotes the length of the cycle $\gamma$. 
\end{theorem}

\begin{proof}
    We prove this theorem using induction on the number of edges that are not loops in $X$. If $X$ does not have such edges, then $W_X=W_{D_n}$ since loops do not affect our function. According to Example \ref{diskretan}, we have that \begin{equation} \label{nesto}
        W_X=\sum_{\pi\in\Pi_n} \pi! m_{\pi}.\end{equation}
    On the other hand, any permutation of vertices of $X$ is in the set $\mathbb{S}_V(X, \overline{X})$ due to the fact that $\overline{X}$ is the complete digraph on $n$ vertices. Also, $\varphi(\sigma)=0$ for any $\sigma$. Therefore, this theorem says that \[W_X=\sum_{\sigma\in\mathbb{S}_n}p_{\mathrm{Type}(\sigma)}.\]
    We claim that this expansion is the same as the one in Equation (\ref{nesto}). Using Equation (\ref{pprekom}), we see that \[[m_{\pi}]\sum_{\sigma\in\mathbb{S}_n}p_{\mathrm{Type}(\sigma)}=\#\{\sigma\in\mathbb{S}_n\mid \mathrm{Type}(\sigma)\leq\pi\}=\pi!,\]
    which proves the base of our induction.

    Now, if $X$ has edges that are not loops, we use the relabeling proposition to obtain a labeling of $X$ such that $e=(v_{n-1}, v_n)$ is an edge of $X$. From the deletion-contraction property, we have that \[W_X=W_{X\setminus e}-W_{X/e}\uparrow.\] 
    Since digraphs $X\setminus e$ and $X/e$ have less edges that are not loops than $X$, by inductive hypothesis, we only need to prove that \begin{multline}
        \label{dovoljno}
        \sum_{\sigma\in\mathbb{S}_V(X, \overline{X})}(-1)^{\varphi(\sigma)}p_{\mathrm{Type}(\sigma)}=\\\sum_{\sigma\in\mathbb{S}_V(X\setminus e, \overline{X\setminus e})}(-1)^{\varphi(\sigma)}p_{\mathrm{Type}(\sigma)}- \sum_{\sigma\in\mathbb{S}_V(X/e, \overline{X/e})}(-1)^{\varphi(\sigma)}p_{\mathrm{Type}(\sigma)}\uparrow. 
        \end{multline}
        
    Let $\sigma\in\mathbb{S}_V(X,\overline{X})$. If there is a cycle of $\sigma$ of the form $(\ldots, u, v_{n-1}, v_n, w, \ldots)$, then $\sigma$ induces a unique $\widehat{\sigma}\in\mathbb{S}_V(X/e, \overline{X/e})$ obtained from $\sigma$ by replacing that cycle with the cycle $(\ldots, u, e, w, \ldots)$. Since $v_{n-1}v_n\in E$, these cycles are in $X$ and $X/e$, hence $\varphi(\sigma)=\varphi(\widehat{\sigma})+1$. Clearly, $p_{\mathrm{Type}(\sigma)}=p_{\mathrm{Type}(\widehat{\sigma})}\uparrow$, and therefore $(-1)^{\varphi(\sigma)}p_{\mathrm{Type}(\sigma)}=-(-1)^{\varphi(\widehat{\sigma})}p_{\mathrm{Type}(\widehat{\sigma})}\uparrow$.

    Next, if there is not a cycle of $\sigma$ of the form $(\ldots, u, v_{n-1}, v_n, w, \ldots)$, then $\sigma\in\mathbb{S}_V(X\setminus e, \overline{X\setminus e})$. Hence, we only need to prove that the terms of the right side of \ref{dovoljno} that do not appear in these two cases cancel each other out. 

    If $\sigma\in\mathbb{S}_V(X/e, \overline{X/e})$ does not correspond to any listing from $\mathbb{S}_V(X, \overline{X})$, then there is a cycle of $\sigma$ of the form $(\ldots, u, e, w, \ldots)$ that is in $\overline{X/e}$. Hence, $(u, e)$ and $(e, w)$ are edges of $\overline{X/e}$. Therefore, $(u, e)$ and $(e, w)$ are not edges of $X/e$. Consequently, $(u, v_{n-1})$ and $(v_n, w)$ are not edges of $X\setminus e$, hence, they are edges of its complement. In other words, from the original cycle, we can obtain the cycle $(\ldots, u, v_{n-1}, v_n, w, \ldots)$ in $\overline{X\setminus e}$, which transforms $\sigma\in\mathbb{S}_V(X/e, \overline{X/e})$ into $\sigma'\in\mathbb{S}_V(X\setminus e, \overline{X\setminus e})$.  Since these cycles are in $\overline{X/e}$ and $\overline{X\setminus e}$, $\varphi(\sigma)=\varphi(\sigma')$ and $p_{\mathrm{Type}(\sigma)}\uparrow=p_{\mathrm{Type}(\sigma')}$, these terms cancel each other out in Equation (\ref{dovoljno}). Similarly, it is easy to see that all the listings from $\mathbb{S}_V(X\setminus e, \overline{X\setminus e})$ that do not correspond to any listing from $\mathbb{S}_V(X, \overline{X})$ are of the form $\sigma'$ for some $\sigma\in\mathbb{S}_V(X/e, \overline{X/e})$, which completes the proof.

\end{proof}

\begin{corollary}
     $W_X$ has integer coefficients in its expansion in power sum, and therefore, in monomial bases.  If $X$ does not have cycles of even length, then the coefficients of $W_X$ in power sum basis are positive integers.
\end{corollary}

The proof of the following corrolary is the same as the proof of Corollary \ref{turnir}, so is left out.

\begin{corollary}
    If $X=(V, E)$ is a tournament, then \[W_X=\sum_{\substack{\sigma\in \mathbb{S}_V(X)\\ \textrm{all cycles of }\sigma \textrm{ have odd length}}}2^{\psi(\sigma)}p_{\mathrm{Type}(\sigma)},\]
    where $\psi(\sigma)$ denotes the number of nontrivial cycles of $\sigma$.
\end{corollary}

 Combining the expansion from Theorem \ref{razvoj} with Equations (\ref{pprekom}) and (\ref{pprekoe}) gives us the following.

\begin{corollary} If $X$ is a digraph, then
    \[[m_\pi]W_X=\sum_{\substack{\sigma\in \mathbb{S}_V(X, \overline{X})\\\mathrm{Type}(\sigma)\leq\pi }}(-1)^{\varphi(\sigma)}, \hspace{10mm}[e_{\pi}]W_X=\sum_{\substack{\sigma\in \mathbb{S}_V(X, \overline{X})\\ \pi\leq \mathrm{Type}(\sigma) }}\frac{(-1)^{\varphi(\sigma)}\mu(\pi, \mathrm{Type}(\sigma))}{\mu(\widehat{0}, \mathrm{Type}(\sigma))}.\]
\end{corollary}

There is a way to obtain a decomposition of $W_X$ analogous to the one from Theorem \ref{razbijanje}. We only need to be more careful while counting the listings.

\begin{lemma}\label{lema}
    If $X=(V, E)$ is a digraph and $F\subseteq E$ such that $(V, F)$ is not a disjoint union of paths and $f: V\rightarrow\mathbb{P}$, then \begin{equation}
        \label{listinzi}
   \#\Sigma_V(f, X)=\sum_{S\subseteq F, S\neq\emptyset}(-1)^{|S|-1}\#\Sigma_V(f, X\setminus S). \end{equation}
\end{lemma}

\begin{proof}
    Since the term on the left side of Equation (\ref{listinzi}) corresponds to $X\setminus S$ for $S=\emptyset$, the statement of this lemma is equivalent to \begin{equation}\label{ekvivalentno}
        \sum_{\substack{S\subseteq F\\2\mid |S|}}\#\Sigma_V(f, X\setminus S)=\sum_{\substack{S\subseteq F\\2\nmid |S|}}\#\Sigma_V(f, X\setminus S). \end{equation}
    Let $\sigma=(\sigma_1, \sigma_2, \ldots, \sigma_n)$ be a listing of $V$. Since $(X, F)$ is not a disjoint union of paths, there is an edge $e$ in $F$ such that $(\sigma_i, \sigma_{i+1})\neq e$ for every $i\in[n-1]$. For $S\subseteq F$, let $S'$ denote $S\setminus e$ if $e\in S$ and $S\cup\{e\}$ if $e\notin S$. This is a self-inverse bijection between the subsets of $F$ with an even number of elements and the subsets of $F$ with an odd number of elements. If $\sigma$ appears in the terms of the left side of Equation (\ref{ekvivalentno}) that correspond to $S_1, S_2, \ldots, S_k$, then $\sigma$ appears in the terms of the right side of 
 Equation (\ref{ekvivalentno}) that correspond to $S_1', S_2', \ldots, S_k'$ and vice-versa, which completes the proof.
\end{proof}

\begin{theorem}
    If $X=(V, E)$ is a labeled digraph that is not a disjoint union of paths, then\begin{equation} \label{podskupovi}
        W_X=\sum_{S\subseteq E, S\neq \emptyset}(-1)^{|S|-1}W_{X\setminus S}.\end{equation}
\end{theorem}

Hence, the Redei-Berge function of any digraph that is not a disjoint union of paths can be expressed as a linear combination of the Redei-Berge functions of appropriate subdigraphs. The same expansion can be applied again to any digraph appearing on the right side of Equation (\ref{podskupovi}) that is not a disjoint union of paths. If we continue with this procedure, we will express the Redei-Berge function of the original digraph as a linear combination of the Redei-Berge functions of its spanning subdigraphs that are disjoint unions of paths. 
Note that we did not use the deletion-contraction property to prove this theorem since some complications dealing with the edges that vanish in $X/e$ might appear. If there is a cycle in $X$, Lemma \ref{lema} gives a generalization of Theorem \ref{ciklus}.

\begin{theorem}
    If $e_1, e_2, \ldots, e_k$ is a list of edges that form a directed cycle in a digraph $X=(V, E)$, then \[W_X=\sum_{\substack{S\subseteq \{e_1, e_2, \ldots, e_k\} \\ S\neq\emptyset}}(-1)^{|S|-1}W_{X\setminus S}. \]
\end{theorem}

\begin{example}
   If $e_1, e_2, e_3$ is a list of edges that form a triangle in a digraph $X=(V, E)$, then \[W_X=W_{X\setminus \{e_1\}}+W_{X\setminus\{ e_2\} }+W_{X\setminus\{ e_3\} }-W_{X\setminus\{ e_1, e_2\} }-W_{X\setminus\{ e_2, e_3\} }-W_{X\setminus\{ e_3, e_1\} }+W_{X\setminus\{ e_1, e_2, e_3\} }.\]
\end{example}




\bibliographystyle{model1a-num-names}
\bibliography{<your-bib-database>}







\end{document}